\documentclass[preprint,12pt,3p]{elsarticle}
\usepackage{verbatim}
\usepackage[active,new,old]{correct} 
\usepackage{amsmath,amsthm,amsfonts,amssymb}
\usepackage{algorithm}
\usepackage{algpseudocode}
\usepackage{mathdots}
\usepackage{cases}
\usepackage[hidelinks]{hyperref}
\usepackage{graphics}
\usepackage{graphicx}
\usepackage{caption}
\usepackage{subcaption}
\usepackage{multicol}
\usepackage[normalem]{ulem}
\usepackage{xcolor,sectsty}
\definecolor{astral}{RGB}{46,116,181}
\subsectionfont{\color{astral}}
\sectionfont{\color{astral}}
\newtheorem{theorem}{Theorem}[section]

\newtheorem{definition}[theorem]{Definition}

\definecolor{darkslategray}{rgb}{0.18, 0.31, 0.31}
\definecolor{warmblack}{rgb}{0.0, 0.26, 0.26}
\definecolor{thrdfc}{rgb}{0.36, 0.54, 0.66}
\definecolor{bole}{rgb}{0.55, 0.71, 0.0}

 \journal{...}

\newcommand{\mb}{\mathbb}
\newcommand{\dg}{{\dagger}}

\begin{document}
\begin{frontmatter}
\title{ \textcolor{warmblack}{\bf Hartwig-Spindelb\"ock decomposition of dual complex matrix}}
\author{Aaisha Be$^{\dag, a}$, Debasisha Mishra{$^{\dag, b}$}}
\address{$^{\dag}$Department of Mathematics,\\
 National Institute of Technology Raipur,\\
 Raipur, Chhattisgarh, India.\\
 \textit{E-mail$^a$}: \texttt{aaishasaeed7\symbol{'100}gmail.com}\\
 \textit{E-mail$^b$}: \texttt{dmishra\symbol{'100}nitrr.ac.in.}}

                                 \begin{abstract}
This article introduces the Hartwig-Spindelb\"{o}ck decomposition of dual complex matrices. We provide representations of some generalized inverses using this decomposition. Further, several characterizations are established for a complex matrix to be Hermitian, normal and new dual EP matrix.
\end{abstract}
                            \begin{keyword}
Hartwig-Spindelb\"{o}ck decomposition, Dual matrix, Dual SVD, NDMPI.                            \vspace{0.2cm}\\
{\it Mathematics Subject Classification}: 15A09, 15A23, 15A66.
 \end{keyword} 
\end{frontmatter}
 \newpage
                                          \section{Introduction}\label{sec:intro}  
In 1873, William Clifford \cite{clifford1873} introduced the notion of dual numbers. Dual numbers and dual matrices are significantly used in kinematic analysis, robotics, rigid body motion, computers, synthesis, mechanics, etc. (see \cite{daniilidis1998, fischer2017, pennestri2009, pennestri2018, wang2012}). A {\it{dual number}} $\hat{a}\in \mb{D}$ is represented as $\hat{a}=a_s+ a_d\epsilon,~ \text{where}~ a_s,~a_d \in \mb{R}$ are called the {\it{standard}} and {\it{dual (or infinitesimal)}} part of $\hat{a}$, respectively and $\epsilon$ is an infinitesimal unit such that $\epsilon^2=0$. $\hat{a}$ is called a {\it{dual complex number (dual quaternion number)}} if $a_s,~a_d\in \mb{C}(\mb{Q})$. A matrix having entries as dual numbers is called a {\it{dual matrix}}. Dual complex and dual quaternion matrices are also similarly defined.
Matrix decomposition is an essential tool in matrix theory. Generally, we decompose a matrix for easy computation purposes \cite{trefthen1997}. Using different matrix decompositions, several beautiful representations of generalized inverses are given \cite{israel2003}. In particular, there is plenty of research on the representations of generalized inverses and their applications \cite{baksalary2008, baksalary2009, malik2014} via the Hartwig-Spindelb\" ock decomposition. This decomposition was introduced by Hartwig and Spindelb\" ock \cite{hartwig1983} in 1983. It also gives numerous characterizations to Hermition, normal, EP matrices, etc.  Many researchers have recently been attracted by dual matrix decompositions and have made tremendous attempts in this direction \cite{cui2025, qi2023a, qi2023b, wang2024, xu2024.NFAO, xu2024.AML, xu2024.CAM} with the applicability of these decompositions in different areas.
This paper aims to introduce the Hartwig-Spindelb\" ock decomposition of dual complex matrices. Using this decomposition, this paper also represents different generalized inverses of dual complex matrices, like Moore-Penrose inverse, group inverse, etc. 

The {\it inverse} of a dual complex matrix  $\hat{A}=A_s+ A_d\epsilon\in \mb{DC}^{n \times n}$ with invertible $A_s\in \mb{C}^{n \times n}$ is defined as 
 $$\hat{A}^{-1}=A_s^{-1}-A_s^{-1}A_dA_s^{-1}\epsilon.$$ 
 Pennestri and Valentini introduced it \cite{pennestri2009} in 2009. Using the Moore-Penrose inverse \cite{israel2003} of the standard part $A_s$ of a dual matrix $\hat{A}=A_s+ A_d\epsilon\in\mb{DC}^{m\times n}$, Pannestri and Stefanelli \cite{pennestri2007} defined the {\it Moore-Penrose dual generalized inverse (MPDGI)} in 2007.
 For any $\hat{A}=A_s+ A_d\epsilon\in\mb{DC}^{m\times n}$, 
$$\hat{A}^P=A_s^{\dg}- A_s^{\dg}A_d A_s^{\dg}\epsilon,$$ 
where $A_s^{\dg}$ is the Moore-Penrose inverse \cite{israel2003} of $A_s$.
Then, in 2018, de Falco {\it{et al.}} \cite{falco2018} introduced the {\it dual Moore-Penrose generalized inverse (DMPGI)}. For $\hat{A}\in \mb{DC}^{m \times n},$ if there exists an $\hat{X}\in \mb{DC}^{n \times m}$ satisfying
$$\hat{A}\hat{X}\hat{A}=\hat{A},~\hat{X}\hat{A}\hat{X}=\hat{X},~(\hat{A}\hat{X})^*=\hat{A}\hat{X},~ \text{and}~(\hat{X}\hat{A})^*=\hat{X}\hat{A},$$
then $\hat{X}$ is called the DMPGI of $\hat{A}$. It is denoted by $\hat{A}^{\dg}$. For the existence of the DMPGI, the authors studied several necessary and sufficient conditions, as the DMPGI may not exist for every dual matrix. After that, in 2025, Cui and Qi \cite{cui2025} introduced the {\it new dual Moore-Penrose inverse (NDMPI)}, which exists for every dual matrix (see Section \ref{sec:Preliminaries} for more details). In 2022, Zhong and Zhang \cite{zhong2022} established the {\it dual group generalized inverse (DGGI)}. The definition of DGGI is stated next. For $\hat{A} \in \mb{DC}^{n \times n}$, if there exists a dual matrix $\hat{X} \in \mb{DC}^{n \times n}$ satisfying
$$\hat{A}\hat{X}\hat{A}=\hat{A}, \hat{X}\hat{A}\hat{X}=\hat{X}, \text{ and } \hat{A}\hat{X}=\hat{X}\hat{A},$$
then $\hat{A}$ is called a DGGI invertible matrix, and $\hat{X}$ is the DGGI of $\hat{A},$ which is denoted as $\hat{A}^{\#}.$ \\
This article aims to study the Hartwig-Spindelb\" ock decomposition for dual complex matrices. To do this, the rest of the article's structure is as follows. The next section provides the background knowledge of dual numbers and matrices. In Section \ref{thm:DHS dec}, the Hartwig-Spindelb\" ock decomposition is introduced. Then, it is used to give new representations of some generalized inverses.   
Also, some necessary and sufficient conditions are provided for a dual complex matrix to be Hermitian, normal, new dual EP, etc. Section \ref{sec:cncln} concludes the main findings of the article.  
                                       \section{Preliminaries}\label{sec:Preliminaries}
This section provides some basic terminologies and techniques of dual matrices that will be used throughout this article.
$\mb{R},~\mb{C},~\mb{Q},~\mb{D},~\mb{DC}$ and $\mb{DQ}$ denote the set of all real numbers, complex numbers, quaternion numbers, dual numbers, dual complex numbers and dual quaternion numbers, respectively. For $S\in \{\mb{R},~\mb{C},~\mb{Q},~\mb{D},~\mb{DC},~\mb{DQ}\}$, $S^{m \times n}$ denotes the set of all matrices of order $m \times n$ whose entries are coming from $S.$ A dual number is called {\it{appreciable}} if its standard part is nonzero; otherwise, it is called {\it{infnitesimal}}. In 2022, Qi {\it{et al.}} \cite{qi2022} defined the total order $``\leq"$ over $\mb{D}$ as follows: For $\hat{a}=a_s+ a_d\epsilon~\text{and}~\hat{b}=b_s+ b_d\epsilon \in \mb{D}$, we have
\begin{itemize}
\item $\hat{a}=\hat{b}$ if $a_s=b_s$ and $a_d=b_d,$
\item $\hat{a}<\hat{b}$ if $a_s<b_s$, or $a_s=b_s$ and $a_d< b_d,$
\item $\hat{a}\leq \hat{b}$ if $a_s<b_s$, or $a_s=b_s$ and $a_d\leq b_d.$
\end{itemize}
Therefore, a dual number $\hat{a}=a_s+ a_d\epsilon\in \mb{D}$ is called positive, nonnegative, negative and nonpositive if $\hat{a}>0,~\hat{a}\geq0,~\hat{a}<0,~\text{and}~\hat{a}\leq 0,$ respectively. For $\hat{a}=a_s+ a_d\epsilon\in \mb{DC},$ the conjugate tranpose of $\hat{a}$, powers of $\hat{a}$, the square root of $\hat{a}$ and the modulus of $\hat{a}$ have the following definitions respectively.
\begin{itemize}
\item $\hat{a}^*=\overline{\hat{a}}=\overline{a}_s+ \overline{a}_d\epsilon$, where $\overline{a}_s$ and $\overline{a}_d$ are the complex conjugates of $a_s$ and $a_d$, respectively.
\item $\hat{a}^n=a_s^n+na_s^{n-1}a_d\epsilon,$ for any positive integer $n.$ \label{itm:power of a}
\item If $\hat{a}$ is positive and appreciable, then $\sqrt{\hat{a}}=\sqrt{a_s}+\frac{a_d}{2\sqrt{a_s}}\epsilon.$
And if $\hat{a}=0$, then $\sqrt{\hat{a}}=0.$ 
\item $|\hat{a}|=\begin{cases}
|a_s|+\frac{a_s\overline{a}_d+{a_d}\overline{a}_s}{2|a_s|}\epsilon,&\text{if}~ a_s\neq 0,\\
|a_d|\epsilon,&\text{otherwise}.
\end{cases}$
\end{itemize}
A matrix $\hat{A}\in \mb{DC}^{n \times n}$ is called Hermitian, idempotent, normal and dual EP if $\hat{A}^*=\hat{A}$, $\hat{A}^2=\hat{A}$, $\hat{A}\hat{A}^*=\hat{A}^*\hat{A}$ and $\hat{A}\hat{A}^{\dg}=\hat{A}^{\dg}\hat{A}$, respectively.
For a dual complex matrix, the singular value decomposition (SVD) is defined as: 
                                     \begin{theorem}(Theorem $5.2$, \cite{qi2022arXiv})\label{thm:svd dc}\\
Let $\hat{A}\in \mb{DC}^{m\times n}$. Then, there exist dual complex unitary matrices $\hat{U}\in \mb{DC}^{m\times m}$ and $\hat{V}\in \mb{DC}^{n\times n}$ such that 
$$\hat{A}=\hat{U} \hat{\Sigma}\hat{V}^*=\hat{U}\begin{bmatrix}
\hat{\Sigma}_{t}& O\\
O &O
\end{bmatrix}\hat{V}^*,$$
where $ \hat{\Sigma}_{t}\in \mb{D}^{t\times t}$ is a dual diagonal matrix of the form \\
$$\hat{\Sigma}_{t}=diag(\hat{\mu}_1,\dots,\hat{\mu}_r,\hat{\mu}_{r+1},\dots,\hat{\mu}_t),$$
$r\leq t \leq \min\{m,n\}$, $\hat{\mu}_1\geq\dots\geq \hat{\mu}_r$ are positive appreciable dual numbers and $\hat{\mu}_{r+1}\geq \dots\geq \hat{\mu}_t$ are positive infinitesimal dual numbers. Counting possible multiplicities of the diagonal entries, the form of $\hat{\Sigma}_{t}$ is unique.
\end{theorem}
The matrix $\hat{\Sigma}\in \mb{DC}^{m\times n}$ can also be written as:
$$\hat{\Sigma}=\begin{bmatrix}
\hat{\Sigma}_{t}& O\\
O &O
\end{bmatrix}=\begin{bmatrix}
\hat{\Sigma}_{1}& O\\
O & \hat{\Sigma}_{2}
\end{bmatrix},$$
where $\hat{\Sigma}_{1}=diag(\hat{\mu}_1,\dots,\hat{\mu}_r)\in  \mb{D}^{r\times r}$ and $\hat{\Sigma}_{2}=diag(\hat{\mu}_{r+1},\dots,\hat{\mu}_t,0,\dots,0)\in  \mb{D}^{m-r\times n-r}.$ Using this representation of $\hat{\Sigma}$ the {\it{essential part}} and the {\it{nonessential part}} \cite{cui2025} of the matrix $\hat{A}\in \mb{DC}^{m\times n}$ are defined as 
$$A_e=\hat{U}\begin{bmatrix}
\hat{\Sigma}_{1}& O\\
O & O
\end{bmatrix}\hat{V}^*=\hat{U}\begin{bmatrix}
\Sigma_{1s}+\Sigma_{1d}\epsilon & O\\
O & O
\end{bmatrix}\hat{V}^*~\text{and} ~A_n=\hat{U}\begin{bmatrix}
O & O\\
O & \hat{\Sigma}_{2}
\end{bmatrix}\hat{V}^*=\hat{U}\begin{bmatrix}
O & O\\
O & \Sigma_{2d}
\end{bmatrix}\hat{V}^*\epsilon.$$
Therefore, $\hat{A}$ also has the representation, $\hat{A}=A_e+A_n.$ 
In 2025, Cui and Qi \cite{cui2025} introduced the NDMPI of a dual complex matrix as follows:
\begin{definition}(Definition $4.1$, \cite{cui2025})\label{def:ndmpi}\\
Let $\hat{A}\in \mb{DC}^{m\times n}.$ Then, the matrix $\hat{X}\in \mb{DC}^{n\times m}$ is called the NDMPI of $\hat{A}$ if it satisfies  the following four equations:
$$\hat{A}\hat{X}\hat{A}=A_e,~\hat{X}\hat{A}\hat{X}=\hat{X},~(\hat{A}\hat{X})^*=\hat{A}\hat{X},~ \text{and}~(\hat{X}\hat{A})^*=\hat{X}\hat{A}.$$
It is denoted by $\hat{A}^N.$ 
\end{definition}
Through the SVD of a dual complex matrix, the NDMPI can be expressed as follows:
                                       \begin{theorem}(Theorem $4.1$, \cite{cui2025})\label{thm:ndmpi dc}\\
 Let $\hat{A}\in \mb{DC}^{m\times n}$ and 
 $$\hat{A}
=\hat{U}\begin{bmatrix}
\hat{\Sigma}_{1}& O\\
O & \hat{\Sigma}_{2}
\end{bmatrix}\hat{V}^*,$$
be the SVD of $\hat{A}$. Then, the NDMPI of $\hat{A}$ is
$$\hat{A}^N=\hat{V}\begin{bmatrix}
\hat{\Sigma}_{1}^{-1} & O\\
O &O
\end{bmatrix}\hat{U}^*.$$
\end{theorem}
                                     \section{Hartwig-Spindelb\"ock decomposition of dual complex matrix}\label{sec:dual HSD}
This section produces the main results of the article. We first define the new dual EP matrix. After that, the Hartwig-Spindelb\"ock decomposition for dual complex matrices is introduced. Then, using this decomposition, we provide representations of the NDMPI of a matrix $\hat{A}\in \mb{DC}^{n\times n}$ and the group inverse of $A_e$, which are then used to characterize different classes of dual complex matrices. Here, we first introduce the notion of a new dual EP matrix.
                                                   \begin{definition}\label{def:new DEP}
A matrix $\hat{A}\in \mb{DC}^{n\times n}$ is called a new dual EP matrix if $$\hat{A}\hat{A}^N=\hat{A}^N\hat{A}.$$     
\end{definition}
If $\hat{A}^{\dg}$ exists, then the new dual EP matrix becomes a dual EP matrix \cite{wang2023}. 
The following is the Hartwig-Spindelb\"ock decomposition of dual complex matrices.                
                                                   \begin{theorem}\label{thm:DHS dec}
 Let $\hat{A}\in \mb{DC}^{n\times n}$. Then, 
 $$\hat{A} =\hat{U} \begin{bmatrix}
\hat{\Sigma}_{1}\hat{K} & \hat{\Sigma}_{1}\hat{L}\\
\hat{\Sigma}_{2}\hat{M} & \hat{\Sigma}_{2}\hat{N}
\end{bmatrix}\hat{U}^*,$$
where $\hat{K}\hat{K}^*+\hat{L}\hat{L}^*=I_r$, $\hat{K}\hat{M}^*+\hat{L}\hat{N}^*=O,$ 
$\hat{\Sigma}_{1}=diag(\hat{\mu}_1,\dots,\hat{\mu}_r)\in  \mb{D}^{r\times r},$ $\hat{\mu}_1\geq\dots\geq \hat{\mu}_r> 0,$ are appreciable dual numbers and  $\hat{\Sigma}_{2}=diag(\hat{\mu}_{r+1},\dots,\hat{\mu}_t,0\dots,0)\in  \mb{D}^{n-r\times n-r},$ $\hat{\mu}_{r+1}\geq \dots\geq \hat{\mu}_t>0$ are infinitesimal dual numbers. Furthermore, 
 $$A_e =\hat{U} \begin{bmatrix}
\hat{\Sigma}_{1}\hat{K} & \hat{\Sigma}_{1}\hat{L}\\
O & O
\end{bmatrix}\hat{U}^*.$$
\end{theorem}  
\begin{proof}
By the SVD of $\hat{A}$, we have 
$$\hat{A}=\hat{U}\begin{bmatrix}
            \hat{\Sigma}_{1}& O\\
            O & \hat{\Sigma}_{2}
        \end{bmatrix}\hat{V}^*.$$ 
Let $\hat{W}=\hat{V}^*\hat{U}=\begin{bmatrix}
    \hat{K} & \hat{L}\\
    \hat{M} & \hat{N}
    \end{bmatrix}.$ Then,  
$$\hat{A}=\hat{U}\begin{bmatrix}
            \hat{\Sigma}_{1}& O\\
            O & \hat{\Sigma}_{2}
\end{bmatrix}\hat{V}^*\hat{U}\hat{U}^*=\hat{U}\begin{bmatrix}
            \hat{\Sigma}_{1}& O\\
            O & \hat{\Sigma}_{2}
\end{bmatrix}\begin{bmatrix}
    \hat{K} & \hat{L}\\
    \hat{M} & \hat{N}
\end{bmatrix}\hat{U}^*=\hat{U}\begin{bmatrix}
    \hat{\Sigma}_{1}\hat{K} & \hat{\Sigma}_{1}\hat{L}\\
   \hat{\Sigma}_{2} \hat{M} & \hat{\Sigma}_{2}\hat{N}
\end{bmatrix}\hat{U}^*$$ and 
    $$A_e=\hat{U}\begin{bmatrix}
            \hat{\Sigma}_{1}& O\\
            O & O
\end{bmatrix}\hat{V}^*\hat{U}\hat{U}^*=\hat{U}\begin{bmatrix}
            \hat{\Sigma}_{1}& O\\
            O & O
        \end{bmatrix}\begin{bmatrix}
    \hat{K} & \hat{L}\\
    \hat{M} & \hat{N}
\end{bmatrix}\hat{U}^*=\hat{U}\begin{bmatrix}
    \hat{\Sigma}_{1}\hat{K} & \hat{\Sigma}_{1}\hat{L}\\
   O & O
\end{bmatrix}\hat{U}^*.$$
Also, we have $\hat{K}\hat{K}^*+\hat{L}\hat{L}^*=I_r$ and $\hat{K}\hat{M}^*+\hat{L}\hat{N}^*=O$ since $\hat{W}$ is unitary. 
\end{proof}
The next theorem is an application of the above decomposition.
                                                                 \begin{theorem}
Let $\hat{A}\in \mb{DC}^{n\times n}$. Then,
\begin{enumerate}[(i)]
\item $\hat{A}^N=\hat{U} \begin{bmatrix}
   \hat{K}^* \hat{\Sigma}_{1}^{-1} & O\\
   \hat{L}^*\hat{\Sigma}_{1}^{-1}  & O
\end{bmatrix}\hat{U}^*;$
\item $A_e^{\#}=\hat{U} \begin{bmatrix}
    \hat{K}^{-1} \hat{\Sigma}_{1}^{-1}&  \hat{K}^{-1} \hat{\Sigma}_{1}^{-1} \hat{K}^{-1}\hat{L}\\
    O & O
\end{bmatrix}\hat{U}^*,$
\end{enumerate}
 where $\hat{U},~\hat{\Sigma}_{1}$ and $\hat{\Sigma}_{2}$ are defined as in Theorem \ref{thm:DHS dec}.
\end{theorem}
\begin{proof}
From the above theorem, we have 
$$\hat{A} =\hat{U} \begin{bmatrix}
    \hat{\Sigma}_{1}\hat{K} & \hat{\Sigma}_{1}\hat{L}\\
    \hat{\Sigma}_{2}\hat{M} & \hat{\Sigma}_{2}\hat{N}
\end{bmatrix}\hat{U}^*~\text{and}~A_e =\hat{U} \begin{bmatrix}
    \hat{\Sigma}_{1}\hat{K} & \hat{\Sigma}_{1}\hat{L}\\
    O & O
\end{bmatrix}\hat{U}^*.$$ 
\begin{enumerate}[(i)]
\item Suppose that $\hat{X}=\hat{U} \begin{bmatrix}
   \hat{K}^* \hat{\Sigma}_{1}^{-1} & O\\
   \hat{L}^*\hat{\Sigma}_{1}^{-1}  & O
\end{bmatrix}\hat{U}^*.$ Then,
\begin{align*}
\hat{A}\hat{X}\hat{A}&=\hat{U}\begin{bmatrix}
    \hat{\Sigma}_{1}\hat{K} & \hat{\Sigma}_{1}\hat{L}\\
   \hat{\Sigma}_{2} \hat{M} & \hat{\Sigma}_{2}\hat{N}
\end{bmatrix}\begin{bmatrix}
   \hat{K}^* \hat{\Sigma}_{1}^{-1} & O\\
   \hat{L}^*\hat{\Sigma}_{1}^{-1}  & O
\end{bmatrix}\begin{bmatrix}
    \hat{\Sigma}_{1}\hat{K} & \hat{\Sigma}_{1}\hat{L}\\
   \hat{\Sigma}_{2} \hat{M} & \hat{\Sigma}_{2}\hat{N}
\end{bmatrix}\hat{U}^*\\
    &= \hat{U}\begin{bmatrix}
\hat{\Sigma}_{1}(\hat{K}\hat{K}^*+\hat{L}\hat{L}^*)\hat{\Sigma}_{1}^{-1}  & O\\
  \hat{\Sigma}_{2}(\hat{M}\hat{K}^*+\hat{N}\hat{L}^*)\hat{\Sigma}_{1}^{-1}   & O
\end{bmatrix} \begin{bmatrix}
    \hat{\Sigma}_{1}\hat{K} & \hat{\Sigma}_{1}\hat{L}\\
   \hat{\Sigma}_{2} \hat{M} & \hat{\Sigma}_{2}\hat{N}
    \end{bmatrix} \hat{U}^*\\
    &=\hat{U}\begin{bmatrix}
I_r & O\\
  O & O
\end{bmatrix} \begin{bmatrix}
    \hat{\Sigma}_{1}\hat{K} & \hat{\Sigma}_{1}\hat{L}\\
   \hat{\Sigma}_{2} \hat{M} & \hat{\Sigma}_{2}\hat{N}
\end{bmatrix} \hat{U}^*~~(\text{since}~\hat{K}\hat{K}^*+\hat{L}\hat{L}^*=I_r~\text{and}~\hat{M}\hat{K}^*+\hat{N}\hat{L}^*=O)\\
    &= \hat{U}\begin{bmatrix}
    \hat{\Sigma}_{1}\hat{K} & \hat{\Sigma}_{1}\hat{L}\\
   O & O
\end{bmatrix} \hat{U}^*\\
    &=A_e. 
\end{align*}
Similarly, we can prove the other three conditions of the NDMPI. Thus, we get $X=A^N.$ 
\end{enumerate}
\end{proof}
Using the above theorem, we prove the following characterizations.
                                 \begin{theorem}\label{thm:eqv cond 1}
Let $\hat{A}\in \mb{DC}^{n\times n}$. Then, 
\begin{enumerate}[(i)]
\item $\hat{A}$ is Hermitian if and only if $\hat{K}^*\hat{\Sigma}_1=\hat{\Sigma}_1\hat{K},$  $N_s^*\Sigma_{2d}=\Sigma_{2d}N_s$ and $\hat{L}=(\Sigma_{1s}^{-1}M_s^*\Sigma_{2d})\epsilon.$  
\item $\hat{A}$ is new dual EP if and only if $\hat{L}=O.$
\item $\hat{A}$ is normal if and only if $\hat{L}=O$ and $\hat{\Sigma}_1\hat{K}=\hat{K}\hat{\Sigma}_1$.
\item $\hat{A}^*=\hat{A}^N$ if and only if $\Sigma_{2d}M_s=\Sigma_{2d}N_s=0$ and $\hat{\Sigma}_1=I.$
\item $\hat{A}^N$ is idempotent if and only if $\hat{\Sigma}_1=\hat{K}.$
\item $\hat{A}^N\hat{A}^*=\hat{A}^*\hat{A}^N$ if and only if  $\Sigma_{2d}M_s=O,$ $L_s\Sigma_{2d}N_s=O,$ and $\hat{\Sigma}_1^2\hat{K}=\hat{K}\hat{\Sigma}_1^2$.
\end{enumerate}
\end{theorem}
                                     \begin{proof}
From Theorem \ref{thm:DHS dec}, we have 
$$\hat{A} =\hat{U} \begin{bmatrix}
    \hat{\Sigma}_{1}\hat{K} & \hat{\Sigma}_{1}\hat{L}\\
    \hat{\Sigma}_{2}\hat{M} & \hat{\Sigma}_{2}\hat{N}
 \end{bmatrix}\hat{U}^*~\text{and}~\hat{A}^N=\hat{U} \begin{bmatrix}
   \hat{K}^* \hat{\Sigma}_{1}^{-1} & O\\
   \hat{L}^*\hat{\Sigma}_{1}^{-1}  & O
 \end{bmatrix}\hat{U}^*.$$
\begin{enumerate}[(i)]
\item We have $\hat{A}^*=\hat{U}\begin{bmatrix}
    (\hat{\Sigma}_{1}\hat{K})^* & (\hat{\Sigma}_{2}\hat{M})^*\\
    (\hat{\Sigma}_{1}\hat{L})^* & (\hat{\Sigma}_{2}\hat{N})^*
 \end{bmatrix}\hat{U}^*.$ So, $\hat{A}$ is Hermitian if and only if 
 \begin{eqnarray}
      (\hat{\Sigma}_{1}\hat{K})^* &=& \hat{\Sigma}_{1}\hat{K}\label{eq:1}\\
     (\hat{\Sigma}_{2}\hat{M})^*&=&  \hat{\Sigma}_{1}\hat{L}\label{eq:2}\\
    (\hat{\Sigma}_{1}\hat{L})^*&=& \hat{\Sigma}_{2}\hat{M}\label{eq:3}\\
    (\hat{\Sigma}_{2}\hat{N})^*&=&\hat{\Sigma}_{2}\hat{N}\label{eq:4}.
    \end{eqnarray}
From \eqref{eq:1} and \eqref{eq:4}, we have 
$\hat{K}^*\hat{\Sigma}_1=\hat{\Sigma}_1\hat{K}~\text{and}~N_s^*\Sigma_{2d}=\Sigma_{2d}N_s,$
respectively.
From equations \eqref{eq:2} and \eqref{eq:3}, we have 
$$ (\hat{\Sigma}_{1}\hat{L})(\hat{\Sigma}_{1}\hat{L})^*=(\hat{\Sigma}_{2}\hat{M})^*(\hat{\Sigma}_{2}\hat{M}),$$
which shows that $$\hat{L}\hat{L}^*=O$$ since $\hat{\Sigma}_{1}$ is invertible and all the entries of $\hat{\Sigma}_{2}$ are infinitesimal numbers. From the last equation, we have that 
$$L_s=O.$$
Now, using this in \eqref{eq:2}, we get $$L_d=\Sigma_{1s}^{-1}M_s^*\Sigma_{2d}.$$
Hence, we get the desired result.
\item We have
$$\hat{A}\hat{A}^N=\hat{U}\begin{bmatrix}
    \hat{\Sigma}_{1}\hat{K} & \hat{\Sigma}_{1}\hat{L}\\
    \hat{\Sigma}_{2}\hat{M} & \hat{\Sigma}_{2}\hat{N}
\end{bmatrix}\begin{bmatrix}
   \hat{K}^* \hat{\Sigma}_{1}^{-1} & O\\
   \hat{L}^*\hat{\Sigma}_{1}^{-1}  & O
\end{bmatrix}\hat{U}^*=\hat{U}\begin{bmatrix}
  I_r & O\\
   O  & O
\end{bmatrix}\hat{U}^*$$ and 
 $$\hat{A}^N\hat{A}=\hat{U}\begin{bmatrix}
   \hat{K}^* \hat{\Sigma}_{1}^{-1} & O\\
   \hat{L}^*\hat{\Sigma}_{1}^{-1}  & O
\end{bmatrix}\begin{bmatrix}
    \hat{\Sigma}_{1}\hat{K} & \hat{\Sigma}_{1}\hat{L}\\
    \hat{\Sigma}_{2}\hat{M} & \hat{\Sigma}_{2}\hat{N}
\end{bmatrix}\hat{U}^*=\hat{U}\begin{bmatrix}
   \hat{K}^* \hat{K} & \hat{K}^*\hat{L}\\
   \hat{L}^*\hat{K}  & \hat{L}^* \hat{L}
\end{bmatrix}\hat{U}^*.$$
$\hat{A}$ is new dual EP, i.e., $\hat{A}\hat{A}^N=\hat{A}^N\hat{A}$ if and only if 
\begin{eqnarray}
      \hat{K}^* \hat{K} &=& I_r\label{eq:5}\\
     \hat{K}^*\hat{L}&=&  O\label{eq:6}\\
   \hat{L}^*\hat{K}&=& O\label{eq:7}\\
   \hat{L}^* \hat{L}&=&O\label{eq:8}.
\end{eqnarray}
From \eqref{eq:5}, we have 
    $K_s^*K_s=I_r$ and $K_s^*K_d+K_d^*K_s=0.$ The condition 
$K_s^*K_s=I_r$ implies that $K_s$ is unitary. On pre and post-multiplying $K_s^*K_d+K_d^*K_s=0$ by $K_s$ and $K_s^*$, respectively, we get $K_dK_s^*+K_sK_d^*=0.$ Therefore, we have $\hat{K}$ is unitary. Thus, we have $\hat{L}=O$ from \eqref{eq:6}. Hence, the assertion follows.
\item Using the conditions $\hat{K}\hat{K}^*+\hat{L}\hat{L}^*=I_r$ and $\hat{K}\hat{M}^*+\hat{L}\hat{N}^*=O,$  we have $$\hat{A}\hat{A}^*=\hat{U}\begin{bmatrix}
    \hat{\Sigma}_{1}\hat{K} & \hat{\Sigma}_{1}\hat{L}\\
    \hat{\Sigma}_{2}\hat{M} & \hat{\Sigma}_{2}\hat{N}
\end{bmatrix}\begin{bmatrix}
    \hat{K}^*\hat{\Sigma}_{1} & \hat{M}^*\hat{\Sigma}_{2}\\
    \hat{L}^*\hat{\Sigma}_{1} & \hat{N}^*\hat{\Sigma}_{2}\\
\end{bmatrix}\hat{U}^*=\hat{U}\begin{bmatrix}
     \hat{\Sigma}_{1}^2 & O\\
     O & O 
\end{bmatrix}\hat{U}^*$$ and 
 $$\hat{A}^*\hat{A}=\hat{U}\begin{bmatrix}
    \hat{K}^*\hat{\Sigma}_{1} & \hat{M}^*\hat{\Sigma}_{2}\\
    \hat{L}^*\hat{\Sigma}_{1} & \hat{N}^*\hat{\Sigma}_{2}\\
\end{bmatrix}\begin{bmatrix}
    \hat{\Sigma}_{1}\hat{K} & \hat{\Sigma}_{1}\hat{L}\\
    \hat{\Sigma}_{2}\hat{M} & \hat{\Sigma}_{2}\hat{N}
\end{bmatrix}\hat{U}^*=\hat{U}\begin{bmatrix}
    \hat{K}^*\hat{\Sigma}_{1}^2\hat{K} & \hat{K}^*\hat{\Sigma}_{1}^2\hat{L}\\
    \hat{L}^*\hat{\Sigma}_{1}^2\hat{K} & \hat{L}^*\hat{\Sigma}_{1}^2\hat{L}\\
\end{bmatrix}\hat{U}^*.$$
So, $\hat{A}$ is normal if and only if 
\begin{eqnarray}
      \hat{K}^* \hat{\Sigma}_{1}^2\hat{K} &=& \hat{\Sigma}_{1}^2\label{eq:9}\\
     \hat{K}^*\hat{\Sigma}_{1}^2\hat{L}&=&  O\label{eq:10}\\
   \hat{L}^*\hat{\Sigma}_{1}^2\hat{K}&=& O\label{eq:11}\\
   \hat{L}^* \hat{\Sigma}_{1}^2\hat{L}&=&O\label{eq:12}.
\end{eqnarray}
Pre-multiplying \eqref{eq:10} by $\hat{K}$ and \eqref{eq:12} by $\hat{L}$, and then adding them, we get 
$$(\hat{K}\hat{K}^*+\hat{L}\hat{L}^*)\hat{\Sigma}_{1}^2\hat{L}=O,$$
which implies that $\hat{L}=O.$ So, we get $\hat{K}$ is unitary. By taking the square root of both sides, we get 
$ \hat{K}^* \hat{\Sigma}_{1}\hat{K} = \hat{\Sigma}_{1}.$ Pre-multiplying the last equation by $\hat{K}$, we get 
$\hat{\Sigma}_{1}\hat{K}=\hat{K}\hat{\Sigma}_{1}.$
\item We have that $\hat{A}^*=\hat{A}^N$ if and only if  
\begin{eqnarray}
      \hat{K}^* \hat{\Sigma}_{1} &=& \hat{K}^*\hat{\Sigma}_{1}^{-1}\label{eq:13}\\
(\hat{\Sigma}_{2}\hat{M})^*&=&  O\label{eq:14}\\
  (\hat{\Sigma}_{2}\hat{N})^*&=&  O\label{eq:15}\\
   \hat{L}^* \hat{\Sigma}_{1}&=&\hat{L}^* \hat{\Sigma}_{1}^{-1}\label{eq:16}.
\end{eqnarray}
Pre-multiplying \eqref{eq:13} by $\hat{K}$ and \eqref{eq:16} by $\hat{L}$, and then adding them, we get \\
$$(\hat{K}\hat{K}^*+\hat{L}\hat{L}^*)\hat{\Sigma}_{1}=(\hat{K}\hat{K}^*+\hat{L}\hat{L}^*)\hat{\Sigma}_{1}^{-1},$$
which implies that
$\hat{\Sigma}_{1}^2=I,$ which shows that 
$$\hat{\mu}_i^2=1~\text{for each}~i=1,2,\dots,r.$$ From which we get 
$$\hat{\mu}_i=1~\text{for each}~i=1,2,\dots,r$$
since each $\hat{\mu}_i$ is positive appreciable dual number.
Thus, $\hat{\Sigma}_{1}=I.$ Now, From \eqref{eq:14} and \eqref{eq:15}, we have $\Sigma_{2d}M_s=\Sigma_{2d}N_s=0.$ 
\item We have 
$$(\hat{A}^{N})^2=\begin{bmatrix}
   \hat{K}^*\hat{\Sigma}_{1}^{-1}\hat{K}^*\hat{\Sigma}_{1}^{-1} & O\\
   \hat{L}^*\hat{\Sigma}_{1}^{-1}\hat{K}^*\hat{\Sigma}_{1}^{-1} & O
\end{bmatrix}.$$
So, $\hat{A}^N$ is idempotent, i.e., $(\hat{A}^{N})^2=\hat{A}^N$ if and only if
\begin{eqnarray}
    \hat{K}^*\hat{\Sigma}_{1}^{-1}\hat{K}^*\hat{\Sigma}_{1}^{-1} &=& \hat{K}^*\hat{\Sigma}_{1}^{-1}\label{eq:17}\\
 \hat{L}^*\hat{\Sigma}_{1}^{-1}\hat{K}^*\hat{\Sigma}_{1}^{-1} &=& \hat{L}^*\hat{\Sigma}_{1}^{-1}.\label{eq:18}
\end{eqnarray}
Pre-multiplying \eqref{eq:17} by $\hat{K}$ and \eqref{eq:18} by $\hat{L}$, and then adding them, we obtain \\  
 $$\hat{\Sigma}_{1}=\hat{K}.$$
\item We have 
 $$\hat{A}^N\hat{A}^*=\hat{U}\begin{bmatrix}
   \hat{K}^* \hat{\Sigma}_{1}^{-1} & O\\
   \hat{L}^*\hat{\Sigma}_{1}^{-1}  & O
\end{bmatrix}\begin{bmatrix}
    (\hat{\Sigma}_{1}\hat{K})^* & (\hat{\Sigma}_{2}\hat{M})^*\\
    (\hat{\Sigma}_{1}\hat{L})^* & (\hat{\Sigma}_{2}\hat{N})^*
\end{bmatrix}\hat{U}^*=\hat{U}\begin{bmatrix}
   \hat{K}^* \hat{\Sigma}_{1}^{-1}\hat{K}^* \hat{\Sigma}_{1} & \hat{K}^* \hat{\Sigma}_{1}^{-1}\hat{M}^* \hat{\Sigma}_{2}\\
   \hat{L}^*\hat{\Sigma}_{1}^{-1} \hat{K}^* \hat{\Sigma}_{1}  & \hat{L}^* \hat{\Sigma}_{1}^{-1}\hat{M}^* \hat{\Sigma}_{2}
\end{bmatrix}\hat{U}^*$$ and 
$$\hat{A}^*\hat{A}^N=\hat{U}\begin{bmatrix}
    (\hat{\Sigma}_{1}\hat{K})^* & (\hat{\Sigma}_{2}\hat{M})^*\\
    (\hat{\Sigma}_{1}\hat{L})^* & (\hat{\Sigma}_{2}\hat{N})^*
\end{bmatrix}\begin{bmatrix}
   \hat{K}^* \hat{\Sigma}_{1}^{-1} & O\\
   \hat{L}^*\hat{\Sigma}_{1}^{-1}  & O
\end{bmatrix}\hat{U}^*=\hat{U}\begin{bmatrix}
   \hat{K}^* \hat{\Sigma}_{1} \hat{K}^* \hat{\Sigma}_{1}^{-1}+ \hat{M}^* \hat{\Sigma}_{2}\hat{L}^*\hat{\Sigma}_{1}^{-1} & O\\
   \hat{L}^*\hat{\Sigma}_{1} \hat{K}^* \hat{\Sigma}_{1}^{-1}+ \hat{N}^* \hat{\Sigma}_{2}\hat{L}^*\hat{\Sigma}_{1}^{-1} & O
\end{bmatrix}\hat{U}^*.$$
 Then, $\hat{A}^N\hat{A}^*=\hat{A}^*\hat{A}^N$ if and only if 
\begin{eqnarray}
    \hat{K}^* \hat{\Sigma}_{1}^{-1}\hat{K}^* \hat{\Sigma}_{1}  &=& \hat{K}^* \hat{\Sigma}_{1} \hat{K}^* \hat{\Sigma}_{1}^{-1} + \hat{M}^* \hat{\Sigma}_{2}\hat{L}^*\hat{\Sigma}_{1}^{-1}\label{eq:19}\\
\hat{K}^* \hat{\Sigma}_{1}^{-1}\hat{M}^* \hat{\Sigma}_{2}&=&  O\label{eq:20}\\
  \hat{L}^*\hat{\Sigma}_{1}^{-1} \hat{K}^* \hat{\Sigma}_{1}&=&  \hat{L}^*\hat{\Sigma}_{1} \hat{K}^* \hat{\Sigma}_{1}^{-1}+ \hat{N}^* \hat{\Sigma}_{2}\hat{L}^*\hat{\Sigma}_{1}^{-1}\label{eq:21}\\
   \hat{L}^* \hat{\Sigma}_{1}^{-1}\hat{M}^* \hat{\Sigma}_{2}&=&O\label{eq:22}.
\end{eqnarray}
Pre-multiplying \eqref{eq:20} by $\hat{K}$ and \eqref{eq:22} by $\hat{L}$, and then adding them, we get 
 $\Sigma_{2d}M_s=O.$ Pre-multiplying \eqref{eq:19} by $\hat{K}$ and \eqref{eq:21} by $\hat{L}$, and then adding them, we obtain
  $\hat{\Sigma}_{1}^2\hat{K}=\hat{K}\hat{\Sigma}_{1}^2.$ 
Post multiplying \eqref{eq:21} by $ \hat{\Sigma}_{1}$, we get
$$\hat{L}^*\hat{\Sigma}_{1}^{-1} \hat{K}^* \hat{\Sigma}_{1}^2= \hat{L}^*\hat{\Sigma}_{1} \hat{K}^* + \hat{N}^* \hat{\Sigma}_{2}\hat{L}^*.$$
Using the condition $ \hat{\Sigma}_{1}^2\hat{K}=\hat{K}\hat{\Sigma}_{1}^2$ in the last equation, we get $L_s\Sigma_{2d}N_s=O.$ Hence, the assertion (vi) follows. This completes the proof.
\end{enumerate}     
\end{proof}
Here, we provide sufficient conditions for a dual complex matrix to be Hermitian.
                                                                  \begin{theorem}
Let $\hat{A}\in \mathbb{DC}^{n \times n}$. 
Then, each one of the following conditions implies that $\hat{A}$ is Hermitian
\begin{enumerate}[(i)]
\item $\hat{A}\hat{A}\hat{A}^N=\hat{A}^*;$
\item $\hat{A}\hat{A}^*\hat{A}^N=\hat{A}.$
\end{enumerate}
\end{theorem}
\begin{proof}
By the representation of $\hat{A}$ and $\hat{A}^N$ as in Theorem \ref{thm:DHS dec}, we can find 
 $$\hat{A}\hat{A}\hat{A}^N=\hat{U}\begin{bmatrix}
\hat{\Sigma}_{1} \hat{K} &  O\\
\hat{\Sigma}_{2} \hat{M} & O
\end{bmatrix}\hat{U}^*$$ 
and $$\hat{A}^*=\hat{U}\begin{bmatrix}
    (\hat{\Sigma}_{1}\hat{K})^* & (\hat{\Sigma}_{2}\hat{M})^*\\
    (\hat{\Sigma}_{1}\hat{L})^* & (\hat{\Sigma}_{2}\hat{N})^*
\end{bmatrix}\hat{U}^*.$$
 $\hat{A}\hat{A}\hat{A}^N=\hat{A}^*$ implies that $\hat{L}=\Sigma_{2d}M_s=\Sigma_{2d}N_s=0$ and  $\hat{K}^*\hat{\Sigma}_{1}=\hat{\Sigma}_{1}\hat{K},$ which shows that $\hat{A}$ is Hermitian using Theorem \ref{thm:eqv cond 1} (i).
\end{proof}
Similar to the above theorem, we can prove the following result.
                                           \begin{theorem}
Let $\hat{A}\in \mathbb{DC}^{n \times n}$. Then, each one of the following conditions implies that $\hat{A}$ is normal 
                                      \begin{enumerate}[(i)]
\item $\hat{A}\hat{A}^*\hat{A}^N=\hat{A}^*;$
\item $\hat{A}^N\hat{A}^*\hat{A}=\hat{A}^*;$
\item $A_e\hat{A}^*A_e^{\#}=\hat{A}^*A_e^{\#}A_e;$
\item $A_e\hat{A}^*A_e^{\#}=A_e^{\#}A_e\hat{A}^*.$
\end{enumerate}
\end{theorem}
Several necessary and sufficient conditions are provided here for a dual complex matrix to be normal.
                                                \begin{theorem}
Let $\hat{A}\in \mathbb{DC}^{n \times n}$. Then, the following are equivalent:
                          \begin{enumerate}[(i)]
\item $\hat{A}$ is normal;
\item $\hat{A}^*A_e^{\#}=A_e^{\#}\hat{A}^*;$
\item $A_e\hat{A}^*\hat{A}^N=\hat{A}^NA_e\hat{A}^*;$
\item $A_eA_e^{\#}\hat{A}^*=A_e^{\#}\hat{A}^*A_e;$
\item $\hat{A}^*\hat{A}A_e^{\#}=A_e^{\#}\hat{A}^*\hat{A};$
\item $\hat{A}^*\hat{A}^NA_e^{\#}=A_e^{\#}\hat{A}^*\hat{A}^N.$
\item $\hat{A}^*A_e^{\#}\hat{A}^N=\hat{A}^N\hat{A}^*A_e^{\#}.$
\end{enumerate}
\end{theorem}
The next theorem characterizes new dual $EP$-matrices.
\begin{theorem}
Let $\hat{A}\in \mathbb{DC}^{n \times n}$. Then, the following are equivalent:
                          \begin{enumerate}[(i)]
\item $\hat{A}$ is new dual EP;
\item $\hat{A}\hat{A}^N\hat{A}^*=\hat{A}^*\hat{A}\hat{A}^N;$
\item $\hat{A}^*\hat{A}^N\hat{A}=\hat{A}^N\hat{A}\hat{A}^*;$
\item $\hat{A}^N\hat{A}^N=\hat{A}^N A_e^{\#};$
\item $\hat{A}^N\hat{A}^N=A_e^{\#}\hat{A}^N;$
\item $\hat{A}^N\hat{A}^N=A_e^{\#}A_e^{\#};$
\item $A_e^{\#}\hat{A}^N=A_e^{\#}A_e^{\#};$
\item $\hat{A}^NA_e^{\#}=A_e^{\#}A_e^{\#};$
\item $\hat{A}^NA_e^{\#}=A_e^{\#}\hat{A}^N;$ 
\item $\hat{A}^N\hat{A}^NA_e^{\#}=\hat{A}^NA_e^{\#}\hat{A}^N;$ 
\item $\hat{A}^N\hat{A}^NA_e^{\#}=A_e^{\#}\hat{A}^N\hat{A}^N;$ 
\item $\hat{A}^NA_e^{\#}\hat{A}^N=A_e^{\#}\hat{A}^N\hat{A}^N;$ 
\item $\hat{A}^NA_e^{\#}A_e^{\#}=A_e^{\#}\hat{A}^NA_e^{\#};$ 
\item $\hat{A}^NA_e^{\#}A_e^{\#}=A_e^{\#}A_e^{\#}\hat{A}^N;$
\item $A_e^{\#}A_e^{\#}\hat{A}^N=A_e^{\#}\hat{A}^NA_e^{\#};$
\end{enumerate}
\end{theorem}
\begin{proof}
First, we prove that (i) and (ii) are equivalent.
By the representation of $\hat{A}$ and $\hat{A}^N$ as in Theorem \ref{thm:DHS dec}, we can find 
$$\hat{A}\hat{A}^N\hat{A}^*=\hat{U}\begin{bmatrix}
   \hat{K}^* \hat{\Sigma}_{1} &  \hat{M}^* \hat{\Sigma}_{2}\\
    O & O
\end{bmatrix}\hat{U}^*$$
 and 
 $$\hat{A}^*\hat{A}\hat{A}^N=\hat{U}\begin{bmatrix}
   \hat{K}^* \hat{\Sigma}_{1} &  O\\
    \hat{L}^* \hat{\Sigma}_{1} & O
\end{bmatrix}\hat{U}^*.$$
Therefore, we have $\hat{A}\hat{A}^N\hat{A}^*=\hat{A}^*\hat{A}\hat{A}^N$ $\Leftrightarrow$ $\hat{L}=O$ $\Leftrightarrow$ $\hat{A}$ is a new dual EP matrix. The last equivalence follows from Theorem \ref{thm:eqv cond 1} (ii). Thus, (i)$\Leftrightarrow$(ii). Similarly, we can prove the remaining equivalent conditions.    
\end{proof}
\begin{theorem}
   Let $\hat{A}\in \mathbb{DC}^{n \times n}$ be an EP-matrix. Then,
   $\hat{A}$ is normal if and only if $ \hat{\Sigma}_{1}\hat{K}=\hat{K} \hat{\Sigma}_{1}.$   
\end{theorem}
                                            \section{Conclusion}\label{sec:cncln}
The essential determinations of the article are summarized as follows:
\begin{itemize}
\item Using the singular value decomposition, the Hartwig-Spindelb\" ock decomposition is provided for dual complex matrices, which is then used to give a new representation of the NDMPI of a dual matrix and the group inverse of its essential part.
\item Several necessary and sufficient conditions are provided for a dual complex matrix to be Hermitian, normal, new dual EP.
\end{itemize}
                                          \section*{Acknowledgements}
The first author acknowledges the support of the Council of Scientific and Industrial Research, India. 
                                          \section*{Conflict of Interest}
 The authors declare that there is no conflict of interest.
                                               \section*{Data Availability Statement}
Data sharing is not applicable to this article as no new data is analyzed in this study.
                                       \bibliographystyle{amsplain}
                                             
\end{document}